\documentclass[leqno,11pt]{article}

\usepackage[utf8]{inputenc}
\usepackage[T1]{fontenc}
\usepackage{microtype}

\usepackage[a4paper]{geometry}

\ifdefined\screenview
  \edef\mtht{\the\textheight}
  \edef\mtwd{\the\textwidth}
\fi

\usepackage[dvipsnames]{xcolor}

\usepackage{amsmath}
\usepackage{amsthm}
\usepackage{tikz-cd}
\usetikzlibrary{arrows} 
\tikzset{
  commutative diagrams/.cd, 
  arrow style=tikz, 
  diagrams={>=stealth}
}
\usetikzlibrary{matrix,decorations.pathreplacing,calc}

\usepackage{textcomp}
\usepackage[sb]{libertine}
\usepackage[varqu,varl]{zi4}%
\usepackage[libertine,bigdelims,vvarbb]{newtxmath}
\usepackage[supstfm=libertinesups,supscaled=1.2,raised=-.13em]{superiors}
\useosf

\usepackage[scr=boondox,cal=euler]{mathalfa}

\usepackage{marvosym}

\usepackage{slashed}
\usepackage{esint} 
\usepackage[english]{babel}

\usepackage{imakeidx}
\indexsetup{noclearpage}
\makeindex[intoc]

\usepackage{csquotes}
\usepackage[
  backend=biber,
  hyperref=true,
  backref=true,
  isbn=false,
  doi=true,
  natbib=true,
  eprint=true,
  useprefix=true,
  maxcitenames=99,
  maxbibnames=99,  
  maxalphanames=99, 
  minalphanames=99,
  safeinputenc,
  style=alphabetic,
  citestyle=alphabetic,
  block=space,
  datamodel=preamble/ext-eprint
]{biblatex}
\usepackage[
  bookmarksnumbered = true,
  hypertexnames = false,
  colorlinks    = true,
  citecolor     = gray,
  linkcolor     = gray,
  urlcolor      = gray,
  breaklinks
]{hyperref}

\DeclareFieldFormat{url}{%
  \href{#1}{\ComputerMouse}
}
\DeclareFieldFormat{doi}{%
  \mkbibacro{DOI}\addcolon\space
  \href{https://doi.org/#1}{#1}
}
\makeatletter
\DeclareFieldFormat{arxiv}{%
  arXiv\addcolon\space
  \href{http://arxiv.org/\abx@arxivpath/#1}{#1}
}
\makeatother
\DeclareFieldFormat{mr}{%
  MR\addcolon\space
  \href{http://www.ams.org/mathscinet-getitem?mr=MR#1}{#1}
}
\DeclareFieldFormat{zbl}{%
  Zbl\addcolon\space
  \href{http://zbmath.org/?q=an:#1}{#1}
}
\renewbibmacro*{eprint}{%
  \printfield{arxiv}%
  \newunit\newblock
  \printfield{mr}%
  \newunit\newblock
  \printfield{zbl}%
  \newunit\newblock
  \iffieldundef{eprinttype}
  {\printfield{eprint}}
  {\printfield[eprint:\strfield{eprinttype}]{eprint}}
}

\AtEveryBibitem{%
  \clearlist{publisher}%
}
\AtEveryBibitem{%
  \clearlist{address}%
}
\DeclareFieldFormat[article,inproceedings,inbook,incollection,thesis]{title}{\textit{#1}}
\renewbibmacro{in:}{}
\newcommand{\printreferences}{\printbibliography[heading=bibintoc]}

\usepackage[inline,shortlabels]{enumitem}

\usepackage{subcaption}

\usepackage[yyyymmdd]{datetime}

\usepackage{etoolbox}
\ifundef{\abstract}{}{\patchcmd{\abstract}%
    {\quotation}{\quotation\noindent\ignorespaces}{}{}}

\usepackage[super]{nth}

\usepackage{thmtools}

\numberwithin{equation}{section}

\renewcommand{\qedsymbol}{$\blacksquare$}

\newcommand{\CorollaryQED}{\qedsymbol}
\newcommand{\ConjectureQED}{$\square$}
\newcommand{\SituationQED}{$\times$}
\newcommand{\DefinitionQED}{$\bullet$}
\newcommand{\NotationQED}{$\circ$}
\newcommand{\ExampleQED}{$\spadesuit$}
\newcommand{\RemarkQED}{$\clubsuit$}

\declaretheorem[numberlike=equation,]{theorem}
\declaretheorem[numbered=no,name=Theorem]{theorem*}

\declaretheorem[numberlike=equation,name=Proposition]{prop}

\declaretheorem[numberlike=equation,name=Definition,style=definition,qed=\DefinitionQED]{definition}
\declaretheorem[numbered=no,name=Definition,style=definition,qed=\DefinitionQED]{definition*}

\declaretheorem[numberlike=equation,style=definition,qed=\ExampleQED]{example}

\declaretheorem[numberlike=equation,style=remark,qed=\RemarkQED]{remark}

\def\makeautorefname#1#2{\AtBeginDocument{\expandafter\def\csname#1autorefname\endcsname{#2}}}
\makeautorefname{table}{Table}        
\makeautorefname{chapter}{Chapter}
\makeautorefname{section}{Section}
\makeautorefname{subsection}{Section}
\makeautorefname{subsubsection}{Section}
\makeautorefname{footnote}{Footnote}
\AtBeginDocument{\def\itemautorefname~#1\null{(#1)\null}}
\AtBeginDocument{\def\equationautorefname~#1\null{(#1)\null}}

\newtheorem{step}{Step}

\numberwithin{substep}{step}
\makeautorefname{step}{Step}
\makeautorefname{substep}{Step}

\makeautorefname{case}{Case}
\makeautorefname{substep}{Step}

\setlist[description]{leftmargin=!,labelindent=1em}
\setlist[enumerate]{label={\rm (\arabic*)},ref=\arabic*}
\setlist[enumerate,2]{label={\rm (\alph*)},ref=\theenumi.\alph*}
\let\C\undefined
\let\U\undefined

\usepackage{bm}
\usepackage{mathtools} 
\usepackage{stmaryrd} 

\DeclareFontFamily{U}{mathx}{\hyphenchar\font45}
\DeclareFontShape{U}{mathx}{m}{n}{
      <5> <6> <7> <8> <9> <10>
      <10.95> <12> <14.4> <17.28> <20.74> <24.88>
      mathx10
      }{}
\DeclareSymbolFont{mathx}{U}{mathx}{m}{n}
\DeclareFontSubstitution{U}{mathx}{m}{n}
\DeclareMathAccent{\widecheck}{0}{mathx}{"71}
\DeclareMathAccent{\wideparen}{0}{mathx}{"75}

\DeclareMathOperator{\End}{End}

\DeclareMathOperator{\HF}{\HF}

\DeclareMathOperator{\Hom}{Hom}

\DeclareMathOperator{\ad}{ad}

\DeclareMathOperator{\diag}{diag}

\DeclareMathOperator{\rk}{rk}

\DeclareMathOperator{\tr}{tr}

\DeclarePairedDelimiter\paren{\lparen}{\rparen}

\DeclarePairedDelimiter{\Abs}{\|}{\|}

\DeclarePairedDelimiter{\Inner}{\langle}{\rangle}
\DeclarePairedDelimiter{\abs}{\lvert}{\rvert}
\DeclarePairedDelimiter{\bracket}{\langle}{\rangle}

\DeclarePairedDelimiter{\set}{\lbrace}{\rbrace}
\def\({\left(}
\def\){\right)}
\def\<{\left\langle}
\def\>{\right\rangle}

\newcommand{\CP}{{\C P}}

\newcommand{\C}{{\mathbf{C}}}

\newcommand{\N}{{\mathbf{N}}}

\newcommand{\R}{\mathbf{R}}

\newcommand{\U}{\mathrm{U}}

\newcommand{\Z}{\mathbf{Z}}

\newcommand{\co}{\mskip0.5mu\colon\thinspace}

\newcommand{\defined}[2][\key]{\def\key{#2}\textbf{#2}\index{#1}}
\newcommand{\delbar}{\bar{\del}}

\newcommand{\del}{\partial}

\newcommand{\gl}{\mathfrak{gl}}

\newcommand{\id}{\mathrm{id}}

\newcommand{\inner}[2]{\bracket{#1, #2}}

\newcommand{\iso}{\cong}

\newcommand{\qandq}{\quad\text{and}\quad}

\newcommand{\qand}{\quad\text{and}}

\newcommand{\qwithq}{\quad\text{with}\quad}

\newcommand{\su}{\mathfrak{su}}

\newcommand{\vol}{\mathrm{vol}}

\renewcommand{\det}{\operatorname{det}}

\renewcommand{\epsilon}{\varepsilon}
\renewcommand{\setminus}{{\backslash}}

\renewcommand{\leq}{\leqslant}
\renewcommand{\geq}{\geqslant}

\makeatletter
\renewcommand*\env@matrix[1][*\c@MaxMatrixCols c]{%
  \hskip -\arraycolsep
  \let\@ifnextchar\new@ifnextchar
  \array{#1}}

\renewcommand\xleftrightarrow[2][]{%
  \ext@arrow 9999{\longleftrightarrowfill@}{#1}{#2}}
\newcommand\longleftrightarrowfill@{%
  \arrowfill@\leftarrow\relbar\rightarrow}
\makeatother

\newcommand{\rd}{{\rm d}}

\newcommand{\bk}{{\mathbf{k}}}

\newcommand{\bE}{{\mathbf{E}}}

\newcommand{\bH}{{\mathbf{H}}}

\newcommand{\bK}{{\mathbf{K}}}

\newcommand{\cN}{\mathcal{N}}

\newcommand{\sA}{\mathscr{A}}

\newcommand{\sC}{\mathscr{C}}

\newcommand{\sE}{\mathscr{E}}
\newcommand{\sF}{\mathscr{F}}
\newcommand{\sG}{\mathscr{G}}

\newcommand{\sM}{\mathscr{M}}

\newcommand{\sO}{\mathscr{O}}

\newcommand{\fd}{{\mathfrak d}}

\newcommand{\fm}{{\mathfrak m}}

\newcommand{\fu}{{\mathfrak u}}

\newcommand{\fV}{{\mathfrak V}}

\newcommand{\EBE}{\text{EBE}}
\newcommand{\Hecke}{\text{Hecke}}
\newcommand{\pHecke}{{\widetilde\Hecke}}

\newcommand{\bdelbar}{{\bm\delbar}}
\newcommand{\bvarphi}{{\bm\varphi}}

\author{Siqi He \and Thomas Walpuski}

\title{
  Hecke modifications of Higgs bundles and the extended Bogomolny equation
}

\begin{document}

\maketitle

\begin{abstract}
  We establish a Kobayashi--Hitchin correspondence between solutions of the extended Bogomolny equation with a Dirac type singularity and Hecke modifications of Higgs bundles.
  This correspondence was conjectured by \citet[p. 668]{Witten2015} and plays an important role in the physical description of the the geometric Langlands program in terms of $S$--duality for $\cN = 4$ super Yang--Mills theory in four dimensions.
\end{abstract}

\section{Introduction}
\label{Sec_Introduction}

\citet{Kapustin2007} describe the geometric Langlands program in terms of $S$--duality for $\cN = 4$ super Yang--Mills theory in four dimensions.
At the heart of their description lies the observation that every solution of the Bogomolny equation with a Dirac type singularity on $[0,1]\times \Sigma$ gives rise to a Hecke modification of a holomorphic bundle over the Riemann surface $\Sigma$ via a scattering map construction \cites[Section 9]{Kapustin2007}{Hurtubise1985}.
Moreover, they anticipated that this construction establishes a bijection between a suitable moduli space of singular monopoles and the moduli space of Hecke modifications---similar to the Kobayashi--Hitchin correspondence \cite{Donaldson1985,Donaldson1986,Uhlenbeck1986,Luebke1995}.
Their conjecture has been proved by \citet{Norbury2011}; see also \citet{Charbonneau2011, Mochizuki2017a}.

In a recent article, \citet{Witten2015} elaborates on the physical description of the geometric Langlands program and emphasizes the importance of the relation between solutions to the \emph{extended} Bogomolny equation with a Dirac type singularity on $[0,1]\times \Sigma$ and Hecke modifications of \emph{Higgs} bundles.
While Hecke modifications of holomorphic bundles have been studied intensely for quite some time (see, e.g., \cite{Pressley1986,Zhu2016}),
interest in Hecke modifications of Higgs bundles has only emerged recently.
They do appear, for example, in \citeauthor{Nakajima2017}'s recent work on a mathematical definition of Coulomb branches of $3$--dimensional $\cN = 4$ gauge theories \cite[Section 3]{Nakajima2017}.

The purpose of this article is to
\begin{enumerate*}[(a)]
\item
  give a precise statement of the Kobayashi--Hitchin correspondence conjectured by Witten and
\item
  establish this correspondence.
\end{enumerate*}
The upcoming four sections review the notion of a Hecke modification of a Higgs bundle,
the extended Bogomolny equation,
Dirac type singularities, and
the scattering map construction.
The main result of this article is stated as \autoref{Thm_EBE=Hecke}.
The remaining five sections contain the proof of this result.

Our proof, like \citeauthor{Norbury2011}'s, heavily relies on the work of \citet{Simpson1988}.
However, unlike \citeauthor{Norbury2011},
we cannot make use of the extensive prior work on Dirac type singularities for solutions of the Bogomolny equation \cite{Kronheimer1985,Pauly1998,Mochizuki2017}.
Instead, our singularity analysis is based on ideas from recent work on tangent cones of singular Hermitian Yang--Mills connections \cite{Jacob2016a, Chen2017}.
\autoref{Thm_EBE=Hecke} can be easily generalized to a Kobayashi--Hitchin correspondence between solutions of the extended Bogomolny equation with multiple Dirac type singularities and sequences of Hecke modifications of Higgs bundles.
This result is stated as \autoref{Thm_EBE=Hecke_Sequences} and proved in \autoref{Sec_SequencesOfHeckeModifications}.
Moreover, although we do not provide details here, both of these results can be further generalized to $G^\C$ Higgs bundles by fixing an embedding $G \subset \U(r)$,
see \cite[Proof of Proposition 8.2]{Simpson1988}.

\paragraph{Acknowledgements}

Siqi He thanks Xinwen Zhu, Rafe Mazzeo, Sergey Cherkis, and Song Sun for helpful discussions.
Thomas Walpuski thanks Gorapada Bera and Gonçalo Oliveira for carefully reading a draft of the present article.
This material is based upon work supported by  \href{https://www.nsf.gov/awardsearch/showAward?AWD_ID=1754967&HistoricalAwards=false}{the National Science Foundation under Grant No.~1754967}
and
\href{https://sloan.org/grant-detail/8651}{an Alfred P. Sloan Research Fellowship}.


\section{Hecke modifications of Higgs bundles}
\label{Sec_HeckeModifications}

In this section, we briefly recall the notion of a Hecke modification of a Higgs bundle.
We refer the reader to \cites[Section 4.5]{Witten2015} for a more extensive discussion.
Throughout this section, let $(\Sigma,I)$ be a closed Riemann surface and denote its canonical bundle by $K_\Sigma$.

\begin{definition}
  A \defined{Higgs bundle} over $\Sigma$ is a pair $(\sE,\varphi)$ consisting of a holomorphic vector bundle $\sE$ over $\Sigma$ and a holomorphic $1$--form $\varphi \in H^0(\Sigma,K_\Sigma\otimes\End(\sE))$ with values in $\End(\sE)$.
\end{definition}

Let $(E,H)$ be a Hermitian vector bundle over $\Sigma$.
Given a holomorphic structure $\delbar$ on $E$,
there exists a unique unitary connection $A \in \sA(E,H)$ satisfying
\begin{equation*}
  \nabla_A^{0,1} = \delbar;
\end{equation*}
see, e.g., \cite[Section 6]{Chern1995}.
Furthermore, every $\varphi \in \Omega^{1,0}(\Sigma,\End(E))$ can uniquely be written as
\begin{equation*}
  \varphi = \frac12(\phi - iI\phi)
\end{equation*}
with $\phi \in \Omega^1(\Sigma,\fu(E,H))$.
Here $I$ is the complex structure on $\Sigma$ and $\fu(E,H)$ denotes the bundle of skew-Hermitian endomorphism of $(E,H)$.
It follows from the Kähler identities that $\varphi$ is holomorphic if and only if
\begin{equation}
  \label{Eq_PhiHarmonic}
  \rd_A \phi = 0 \qandq
  \rd_A^* \phi = 0.
\end{equation}

\begin{remark}
  \citet[Theorem 2.1 and Theorem 4.3]{Hitchin1987a} proved that a Higgs bundle $(\sE,\varphi)$ of rank $r \coloneq \rk \sE$ admits a Hermitian metric $H$ such that $(A,\phi)$ satisfies \defined{Hitchin's equation}
  \begin{equation}
    \label{Eq_Hitchin}
    F_A^\circ - \frac12[\phi\wedge \phi] = 0, \quad
    \rd_A \phi = 0, \qandq
    \rd_A^* \phi = 0
  \end{equation}
  if and only if it is $\mu$--polystable.
  Here $F_A^\circ \coloneq F_A- \frac1r\tr(F_A)\id_E$.
  Furthermore,
  if $(\sE,\varphi)$ is $\mu$--stable,
  then imposing the additional condition that $H$ induces a given Hermitian metric on $\det \sE$ makes it unique.
\end{remark}

\begin{definition}
  Let $(\sE,\varphi)$ be a Higgs bundle over $\Sigma$ of rank $r$.  
  Let $z_0 \in \Sigma$ and $\bk = (k_1,\ldots,k_r) \in \Z^r$ satisfying
  \begin{equation}
    \label{Eq_KIncreasing}
    k_1 \leq k_2 \leq \cdots \leq k_r.
  \end{equation}
  A \defined{Hecke modification} of $(\sE,\varphi)$ at $z_0$ of type $\bk$ is a Higgs bundle $(\sF,\chi)$ over $\Sigma$ together with an isomorphism
  \begin{equation*}
    \eta\co (\sE,\varphi)|_{\Sigma\setminus\set{z_0}} \iso (\sF,\chi)|_{\Sigma\setminus\set{z_0}} 
  \end{equation*}
  of Higgs bundles which, in suitable holomorphic trivializations near $z_0$, is given by
  \begin{equation*}
    \diag(z^{k_1},\ldots,z^{k_r}).
  \end{equation*}
  An \defined{isomorphism} between two Hecke modifications $(\sF_1,\chi_1;\eta_1)$ and $(\sF_2,\chi_2;\eta_2)$ of $(\sE,\varphi)$ is an isomorphism
  \begin{equation*}
    \zeta\co (\sF_1,\chi_1) \to (\sF_2,\chi_2)
  \end{equation*}
  such that
  \begin{equation*}
    \eta_1 = \eta_2\zeta.
  \end{equation*}
  We denote by
  \begin{equation*}
    \sM^\Hecke(\sE,\varphi;z_0,\bk)
  \end{equation*}
  the set of all isomorphism classes of Hecke modifications of $(\sE,\varphi)$ at $z_0$ of type $\bk$.
\end{definition}

\begin{remark}
  If $\varphi = 0$, then the above reduces to the classical notion of a Hecke modification of a holomorphic vector bundle.
\end{remark}


\section{Singular solutions of the extended Bogomolny equation}
\label{Sec_SingularEBE}

Throughout this section,
let $M$ be an oriented Riemannian $3$--manifold (possibly with boundary) and
let $(E,H)$ be a Hermitian vector bundle over $M$.

\begin{definition}
  \label{Def_EBE}
  The \defined{extended Bogomolny equation} is the following partial differential equation for $A \in \sA(E,H)$, $\phi \in \Omega^1(M,\fu(E,H))$, and $\xi \in \Omega^0(M,\fu(E,H))$:
  \stepcounter{equation}
  \begin{align*}
    \label{Eq_EBE}
    F_A - \frac12[\phi\wedge\phi] &= *\rd_A\xi, \\
    \tag{\theequation}
    \rd_A\phi - *[\xi,\phi] &= 0, \qand \\
    \rd_A^*\phi &= 0.
                  \qedhere
  \end{align*}
\end{definition}

\begin{remark}
  The extended Bogomolny equation arises from the Kapustin--Witten equation \cite{Kapustin2007} by dimensional reduction.
  It can be thought of as a complexification of the Bogomolny equation.
  In fact, for $\phi = 0$, it reduces to the Bogomolny equation.
\end{remark}

In this article, we are exclusively concerned with singular solutions of \autoref{Eq_EBE}.
The following example is archetypical.

\begin{example}
  \label{Ex_DiracMonopole}
  Let $k \in \Z$.
  The holomorphic line bundle $\sO_{\CP^1}(k) \to \CP^1 \iso S^2$ admits a metric $H_k$ whose associated connection $B_k$ satisfies
  \begin{equation*}
    F_{B_k} = -\frac{ik}{2}\vol_{S^2}.
  \end{equation*}
  Denote by $\pi\co \R^3\setminus\set{0} \to S^2$ the projection map and denote by $r\co \R^3 \to [0,\infty)$ the distance to the origin.

  Given $\bk \in \Z^r$ satisfying \autoref{Eq_KIncreasing},
  set
  \begin{equation*}
    (E_\bk,H_\bk) \coloneq \bigoplus_{i=1}^r \pi^*(\sO_{\CP^1}(k_i),H_{k_i}), \quad
    A_\bk \coloneq \bigoplus_{i=1}^r \pi^*B_k, \qandq
    \xi_\bk \coloneq \frac{1}{2r}\diag(ik_1,\ldots,ik_r).
  \end{equation*}
  The pair $(A_\bk,\xi_\bk)$ is called the \defined{Dirac monopole} of type $\bk$.
  It satisfies the \defined{Bogomolny equation}
  \begin{equation*}
    F_{A_\bk} = *\rd_{A_\bk} \xi_\bk
  \end{equation*}
  and thus \autoref{Eq_EBE} with $\phi = 0$.
\end{example}

Henceforth,
we suppose that $\bar M$ is an oriented Riemannian $3$--manifold, $p \in \bar M$ is an interior point, and $M$ is the complement of $p$ in $\bar M$.
Define $r\co M \to (0,\infty)$ by
\begin{equation*}
  r(x) \coloneq d(x,p).
\end{equation*}
Furthermore,
we fix $\bk \in \Z^r$ satisfying \autoref{Eq_KIncreasing}.

\begin{definition}
  \label{Def_Framing}
  A \defined{framing} of $(E,H)$ at $p$ of type $\bk$ is an isometry of Hermitian vector bundles
  \begin{equation*}
    \Psi\co \exp_p^*(E,H)|_{B_\rho(0)} \to (E_\bk,H_\bk)|_{B_\rho(0)}
  \end{equation*}
  for some $\rho > 0$.
\end{definition}

\begin{definition}
  \label{Def_SingularityPreserving}
  Let $\Psi$ be a framing of $(E,H)$ at $p$ of type $\bk$.
  A solution $(A,\phi,\xi)$ of \autoref{Eq_EBE} on $(E,H)$ is said to have a \defined{Dirac type singularity} at $p$ of type $\bk$ if there exists an $\alpha > 0$ such that for every $k \in \N_0$
  \begin{equation*}
    \nabla_{A_\bk}^k(\Psi_*A - A_\bk) = O(r^{-k-1+\alpha}), \quad
    \nabla_{A_\bk}^k\Psi_*\phi = O(r^{-k}), \qandq
    \nabla_{A_\bk}^k(\Psi_*\xi - \xi_\bk) = O(r^{-k-1+\alpha}).
  \end{equation*}
  A gauge transformation $u \in \sG(E,H)$ is called \defined{singularity preserving} if there exists a $u_p \in \sG(E_\bk,H_\bk)$ satisfying
  \begin{equation*}
    \nabla_{A_\bk} u_p = 0 \qandq
    (u_p)_*\xi_\bk = \xi_\bk
  \end{equation*}
  and an $\alpha > 0$ such that for every $k \in \N_0$
  \begin{equation*}
    \nabla_{A_\bk}^k(\Psi_*u - u_p) = O(r^{-k+\alpha}).
    \qedhere
  \end{equation*}
\end{definition}


\section{The extended Bogomolny equation over $[0,1]\times \Sigma$}
\label{Sec_EBEIxSigma}

Throughout the remainder of this article,
we assume that the following are given:
\begin{enumerate}
\item
  a closed Riemann surface $(\Sigma,I)$,
\item
  a Hermitian vector bundle $(E_0,H_0)$ over $\Sigma$,
\item
  a solution  $(A_0,\phi_0)$ of \autoref{Eq_PhiHarmonic},
\item
  $(y_0,z_0) \in (0,1)\times \Sigma$, and
\item
  $\bk \in \Z^r$ satisfying \autoref{Eq_KIncreasing}.
\end{enumerate}
Set
\begin{equation*}
  M \coloneq [0,1]\times\Sigma \setminus \set{(y_0,z_0)}
\end{equation*}

\begin{prop}
  \label{Prop_ConstructE}
  Given the above data,
  there exists a Hermitian vector bundle $(E,H)$ over $M$ whose restriction to $\set{0}\times\Sigma$ is isomorphic to $(E_0,H_0)$ together with a framing $\Psi$ at $(y_0,z_0)$ of type $\bk$.
  Moreover, any two such $(E,H;\Psi)$ are isomorphic.  
\end{prop}

\begin{proof}
  There is a complex vector bundle $E_1$ over $\Sigma$ together with an isomorphism $\eta\co E_0|_{\Sigma\setminus\set{z_0}} \iso E_1|_{\Sigma\setminus\set{z_0}}$ which can be written as $\diag(z^{k_1},\ldots,z^{k_r})$ in suitable trivializations around $z_0$.
  One can construct $E_1$ and $\eta$, for example, by modifying a Čech cocycle representing $E_0$.
  The complex vector bundle $E$ is now constructed by gluing via $\eta$ the pullback of  $E_0$ to $[0,y_0]\times \Sigma \setminus \set{(y_0,z_0)}$ and the pullback of $E_1$ to $[y_0,1]\times \Sigma \setminus \set{(y_0,z_0)}$.
  Since $E$ is isomorphic near $(y_0,z_0)$ to $E_\bk$,
  we can find the desired Hermitian metric $H$ and framing $\Psi$.
\end{proof}

Henceforth, we fix a choice of
\begin{equation*}
  (E,H;\Psi).
\end{equation*}

\begin{definition}
  \label{Def_MEBE}
  Denote by $\sC^\EBE(A_0,\phi_0;y_0,z_0,\bk)$ the set of triples $A \in \sA(E,H)$, $\phi \in \Omega^1(M,\fu(E,H))$, and $\xi \in \Omega^0(M,\fu(E,H))$ satisfying the extended Bogomolny equation \autoref{Eq_EBE}, as well as
  \begin{equation}
    \label{Eq_IDelYPhi=0}
    i(\del_y)\phi = 0,
  \end{equation}
  and the boundary conditions
  \begin{equation}
    \label{Eq_BoundaryConditions}
    A|_{\set{0}\times\Sigma} = A_0, \quad
    \phi|_{\set{0}\times\Sigma} = \phi_0, \qandq
    \xi|_{\set{1}\times\Sigma} = 0.
  \end{equation}
  Denote by
  \begin{equation*}
    \sG \subset \sG(E,H)
  \end{equation*}
  the subgroup of singularity preserving unitary gauge transformations of $(E,H)$ which restrict to the identity on $\set{0}\times \Sigma$.
  Set
  \begin{equation*}
    \sM^\EBE(A_0,\phi_0;y_0,z_0,\bk)
    \coloneq
    \sC^\EBE(A_0,\phi_0;y_0,z_0,\bk)/\sG.
    \qedhere
  \end{equation*}
\end{definition}

\begin{remark}
  It is an interesting question to ask whether the condition \autoref{Eq_IDelYPhi=0} really does need to be imposed.
  In a variant of our setup on $S^1\times \Sigma$, this condition is automatically satisfied;
  see \cite[Corollary 4.7]{He2017}.
\end{remark}

\begin{remark}
  We refer the reader to \cite[Section 10.1]{Kapustin2007} for a discussion of the significance of the boundary conditions \autoref{Eq_BoundaryConditions}.
  It will become apparent in \autoref{Sec_HermitianMetrics} and \autoref{Eq_XiKes},
  that the boundary conditions on $(A,\varphi,\xi)$ correspond to Dirichlet and Neumann boundary conditions on a Hermitian metric.
\end{remark}

\begin{prop}
  \label{Prop_SplitEBEIntoHolomorphic+Metric}
  Let $A \in \sA(E,H)$, $\phi \in \Omega^1(M,\fu(E,H))$, and $\xi \in \Omega^0(M,\fu(E,H))$ and suppose that \eqref{Eq_IDelYPhi=0} holds.
  Decompose $A$ as
  \begin{equation*}
    \nabla_A = \del_A + \delbar_A + \rd y \wedge \nabla_{A,\del_y}
  \end{equation*}
  and write
  \begin{equation*}
    \phi = \varphi - \varphi^*
    \qwithq
    \varphi \coloneq \frac12(\phi - iI\phi)
    \in \Gamma(\pi_\Sigma^*T^*\Sigma^{1,0}\otimes\End(E)).%
    \footnote{%
      This is possible because of \autoref{Eq_IDelYPhi=0}.
    }
  \end{equation*}
  Set
  \begin{equation*}
    \fd_y \coloneq \nabla_{A,\del_y} - i\xi.
  \end{equation*}
  The extended Bogomolny equation \autoref{Eq_EBE} holds if and only if
  \begin{gather}
    \label{Eq_EBE_Holomorphic}
    \delbar_A\varphi = 0, \quad
    [\fd_y,\delbar_A] = 0, \quad
    \fd_y\varphi = 0, \qandq \\
    \label{Eq_EBE_Metric}
    i\Lambda(F_A + [\varphi\wedge\varphi^*]) - i\nabla_{A,\del_y} \xi = 0.
  \end{gather}
\end{prop}

\begin{proof}
  By the Kähler identities,
  \begin{equation*}
    \rd_A^*\phi
    = i\Lambda(\delbar_A\varphi + \del_A\varphi^*).
  \end{equation*}
  Since $*_\Sigma = -I$,
  $*_\Sigma\varphi = i\varphi$ and thus
  \begin{equation*}
    *\varphi = i\rd y\wedge \varphi.
  \end{equation*}
  Therefore, the second equation of \autoref{Eq_EBE} is equivalent to
  \begin{align*}
    \delbar_A\varphi - \del_A\varphi^* &= 0, \\
    \nabla_{A,\del_y}\varphi - i[\xi,\varphi] &= 0, \qand\\
    \nabla_{A,\del_y}\varphi^* + i[\xi,\varphi^*] &= 0.
  \end{align*}
  This shows that the last two equations of \autoref{Eq_EBE} are equivalent to the first and the last equations of \autoref{Eq_EBE_Holomorphic}.

  We have
  \begin{align*}
    F_A
    &=
      \delbar_A\del_A + \del_A\delbar_A + \rd y \wedge \([\nabla_{A,\del_y},\delbar_A] + [\nabla_{A,\del_y},\del_A]\), \\    
    \frac12[\phi\wedge\phi]
    &=
      -[\varphi\wedge \varphi^*], \qand \\
    *\rd_A\xi
    &=
      \nabla_{A,\del_y}\xi \cdot \vol_\Sigma + i\rd y\wedge \del_A\xi - i \rd y \wedge \delbar_A\xi.
  \end{align*}
  Therefore,
  the first equation of \autoref{Eq_EBE} is equivalent to
  \begin{align*}
    \delbar_A\del_A + \del_A\delbar_A + [\varphi\wedge\varphi^*] - \nabla_{A,\del_y}\xi \cdot\vol_\Sigma &= 0, \\
    [\nabla_{A,\del_y},\del_A] - i\del_A\xi &= 0, \qand \\
    [\nabla_{A,\del_y},\delbar_A] + i\delbar_A\xi &= 0.
  \end{align*}
  These are precisely the second equation in \autoref{Eq_EBE_Holomorphic} as well as \autoref{Eq_EBE_Metric}.
\end{proof}


\section{The scattering map}
\label{Sec_ScatteringMap}

\begin{definition}
  In the situation of \autoref{Ex_DiracMonopole},
  set
  \begin{equation*}
    \delbar_\bk \coloneq \delbar_{A_\bk} \qandq
    \fd_{y,\bk} \coloneq \nabla_{A_\bk,\del_y} - i\xi_\bk.
    \qedhere
  \end{equation*}
\end{definition}

\begin{definition}
  \label{Def_ParametrizedHeckeModification}
  A \defined{parametrized Hecke modification} on $(E,H;\Psi)$ is a triple $(\delbar,\varphi,\fd_y)$ consisting of:
  \begin{enumerate}
  \item
    a complex linear map $\delbar \co \Gamma(E) \to \Gamma(\Hom(\pi_\Sigma^*T\Sigma^{0,1},E))$,
  \item
    a section $\varphi \in \Gamma(\pi_\Sigma^*T^*\Sigma^{1,0}\otimes \End(E))$, and
  \item
    a complex linear map $\fd_y \co \Gamma(E) \to \Gamma(E)$
  \end{enumerate}
  such that the following hold:
  \begin{enumerate}[resume]
  \item
    For every $s \in \Gamma(E)$ and $f \in C^\infty(M,\C)$
    \begin{equation*}
      \delbar (fs) = (\delbar_\Sigma f)\otimes s + f\delbar s \qandq
      \fd_y (fs) = (\del_y f) s + f \fd_y s.
    \end{equation*}
  \item    
    There exists an $\alpha > 0$ such that for every $k \in \N_0$
    \begin{equation}
      \label{Eq_ParametrizedHeckeModification_Estimates}
      \begin{gathered}
        \nabla_{A_\bk}^k(\Psi_*\delbar - \delbar_\bk) = O(r^{-k-1+\alpha}), \quad
        \nabla_{A_\bk}^k\Psi_*\varphi = O(r^{-k}), \qand \\
        \nabla_{A_\bk}^k(\Psi_*\fd_y - \fd_y^\bk) = O(r^{-k-1+\alpha}).
      \end{gathered}
    \end{equation}
  \item
    We have
    \begin{equation}
      \label{Eq_ParametrizedHeckeModification}
      \delbar \varphi = 0, \quad
      [\fd_y,\delbar] = 0, \qandq
      [\fd_y,\varphi] = 0.
      \qedhere
    \end{equation}      
  \end{enumerate}
\end{definition}

The following observation is fundamental to this article.

\begin{prop}[{\citet[Section 9.1]{Kapustin2007}}]
  \label{Prop_ScatteringMap}
  Let $(\delbar,\varphi,\fd_y)$ be a parametrized Hecke modification.
  Denote by $(\sE_0,\varphi_0)$ and $(\sE_1,\varphi_1)$ the Higgs bundles induced by restriction to $\set{0}\times\Sigma$ and $\set{1}\times\Sigma$ respectively.
  The parallel transport associated with the operator $\fd_y$ induces a Hecke modification
  \begin{equation*}
    \sigma\co (\sE_0,\varphi_0)|_{\Sigma\setminus\set{z_0}} \to (\sE_1,\varphi_1)|_{\Sigma\setminus\set{z_0}}
  \end{equation*}
  at $z_0$ of type $\bk$.
\end{prop}

\begin{definition}
  We call $\sigma$ the \defined{scattering map} associated with $(\delbar,\varphi,\fd_y)$.
\end{definition}

For the reader's convenience we recall the proof of \autoref{Prop_ScatteringMap} following \cite{Charbonneau2011}.

\begin{prop}[{\citet[Section 2.2]{Charbonneau2011}}]
  \label{Prop_DiracMonopoleScatteringMap}
  The scattering map for the Dirac monopole of type $\bk$ is given by $\diag(z^{k_1},\ldots,z^{k_r})$ in suitable holomorphic trivializations.
\end{prop}

\begin{proof}
  It suffices to consider the case $r=1$.
  Set
  \begin{equation*}
    U_\pm \coloneq \set{ (y,z) \in \R\times\C: z = 0 \implies \pm y > 0 }.
  \end{equation*}
  There are trivializations $\tau_\pm\co \pi^*\sO_{\CP^1}(k)|_{U_\pm} \iso U_\pm\times \C$ such that the following hold:
  \begin{enumerate}
  \item
    The transition function $\tau \co U_+\cap U_- \to \U(1)$ defined by
    \begin{equation*}
      \tau_+\circ \tau_-^{-1}(y,z;\lambda) \eqcolon (y,z,\tau(y,z)\lambda)
    \end{equation*}
    is given by
    \begin{equation*}
      (y,z) \mapsto (z/\abs{z})^k.
    \end{equation*}
  \item
    The connection $A$ defined in \autoref{Ex_DiracMonopole} satisfies
    \begin{equation*}
      \nabla_{A_\pm}
      \coloneq
        (\tau_\pm)_*\nabla_A
      =
       \rd + \frac{k}{4}(\mp 1 + y/r) \frac{\bar z\rd z - z\rd\bar z}{\abs{z}^2}
     \end{equation*}
     for
     \begin{equation*}
       r \coloneq \sqrt{y^2+\abs{z}^2}.
     \end{equation*}
  \end{enumerate}

  The trivializations $\tau_\pm$ are not holomorphic.
  This can be rectified as follows.
  Since
  \begin{equation*}
    \rd r = \frac{1}{2r}(\bar z \rd z + z\rd\bar z + 2y\rd y),
  \end{equation*}
  the gauge transformations
  \begin{equation*}
    u_\pm(y,z) \coloneq (r\pm y)^{\pm k/2}
  \end{equation*}
  satisfy
  \begin{align*}
    -(\rd u_\pm)u_\pm^{-1}
    &=
      \mp\frac{k}{2(r \pm y)}(\rd r \pm \rd y) \\
    &=        
      \mp\frac{k}{4r(r\pm y)}(\bar z \rd z + z \rd \bar z + 2(y \pm r)\rd y) \\
    &=
      \frac{k}{4}(\mp 1 + y/r)\frac{\bar z \rd z + z \rd \bar z}{\abs{z}^2} - \frac{k}{2r}\rd y.
  \end{align*}
  Therefore,
  \begin{align*}
    \nabla_{\tilde A_\pm}
    &\coloneq 
      (u_\pm)_*\nabla_{A_\pm} \\
    &=
      \nabla_{A_\pm} - (\rd u_\pm)u_\pm^{-1} \\
    &=
      \rd + \frac{k}{2} (\mp 1 + y/r) \frac{\bar z\rd z}{\abs{z}^2} - \frac{k}{2r}\rd y.        
  \end{align*}
  It follows that
  \begin{equation*}
    \delbar_{\tilde A_\pm} = \delbar \qandq
    \nabla_{\tilde A_\pm,\del_y} + \frac{k}{2r} = \del_y.
  \end{equation*}
  Hence, the trivializations $u_\pm\circ\tau_\pm$ are holomorphic and with respect to these the parallel transport associated with $\nabla_{A,\del_y} + \frac{ik}{2r}$ from $y=-\epsilon$ to $y=\epsilon$ is given by
  \begin{equation*}
    u_+(\epsilon,z) \cdot \tau(\epsilon,z) \cdot u_-^{-1}(-\epsilon,z) 
    =
      (r+\epsilon)^{k/2}\(\frac{z}{\abs{z}}\)^k(r-\epsilon)^{k/2}
    =
      z^k.
    \qedhere
  \end{equation*}
\end{proof}

\begin{proof}[Proof of \autoref{Prop_ScatteringMap}]
  The fact that $\sigma$ is holomorphic and preserves the Higgs fields follows directly from \autoref{Eq_ParametrizedHeckeModification}.

  To prove that $\sigma$ is given by $\diag(z^{k_1},\ldots,z^{k_r})$ in suitable trivializations we follow \citet[Proposition 2.5]{Charbonneau2011}.
  It suffices to consider a neighborhood of $(y_0,z_0)$ which we identify with a neighborhood of the origin in $\R\times\C$.
  Since $\fd_y = \fd_{y,\bk} + O(r^{-1+\alpha})$,
  we can construct a section $\tau$ of $\End(E_\bk)$ over $[-\epsilon,0)\times\set{0}$ satisfying  
  \begin{equation}
    \label{Eq_TauIntertwinesDY}
    \fd_y\tau = \tau\fd_{y,\bk} \qandq
    \tau(\cdot,0) = \id_{\C^r} + O(r^\alpha).
  \end{equation}
  First extend $\tau(-\epsilon,0)$ to a section of $\End(E_\bk)$ over $\set{-\epsilon}\times B_\epsilon(0)$ satisfying
  \begin{equation}
    \label{Eq_TauIntertwinesDelbar}
    \delbar\tau = \tau\delbar_\bk
  \end{equation}
  and then further extend it to $[-\epsilon,\epsilon]\times B_\epsilon(0) \setminus [0,\epsilon] \times \set{0}$ by imposing the first part of \autoref{Eq_TauIntertwinesDY}.
  The equation \autoref{Eq_TauIntertwinesDelbar} continues to hold.
  Since $\tau$ is bounded around $(0,0)$,
  it extends to $[-\epsilon,\epsilon]\times B_\epsilon(0)$.
  If $0 < \epsilon \ll 1$, then $\tau$ is invertible.

  By construction, if $\sigma$ denotes the parallel transport associated with $\fd_{y,\bk}$ from $y=-\epsilon$ to $y=\epsilon$,
  then the corresponding parallel transport associated with $\fd_y$ is given by
  \begin{equation*}
    \tau(\epsilon,\cdot) \sigma \tau(-\epsilon,\cdot)^{-1}.
    \qedhere
  \end{equation*}
  In light of \autoref{Prop_DiracMonopoleScatteringMap},
  this proves the assertion.
\end{proof}

The preceding discussion constructs a map
\begin{equation*}
  \sC^\EBE(A_0,\phi_0;y_0,z_0,\bk) \to \sM^\Hecke(\sE_0,\varphi_0;z_0,\bk).
\end{equation*}
This map is $\sG$--invariant.
The following is the main result of this article.

\begin{theorem}
  \label{Thm_EBE=Hecke}
  The map
  \begin{equation*}
    \sM^\EBE(A_0,\phi_0;y_0,z_0,\bk) \to \sM^\Hecke(\sE_0,\varphi_0;z_0,\bk)
  \end{equation*}
  induced by the scattering map construction is bijective.
\end{theorem}

The proof of this theorem occupies the remainder of this article.


\section{Parametrizing Hecke modifications}
\label{Sec_ParametrizingHeckeModifications}

\begin{definition}
  Denote by $(\sE_0,\varphi_0)$ the Higgs bundle induced by $(A_0,\phi_0)$.
  Denote by
  \begin{equation*}
    \sC^\pHecke(\sE_0,\varphi_0;y_0,z_0,\bk)
  \end{equation*}
  the set of parametrized Hecke modifications agreeing with $(\sE_0,\varphi_0)$ at $y=0$.
  Denote by
  \begin{equation*}
    \sG^\C \subset \sG^\C(E)
  \end{equation*}
  the group of singularity preserving complex gauge transformations of $E$ which are the identity at $y=0$.
  Here singularity preserving means the analogue of the condition in \autoref{Def_SingularityPreserving} holds.

  Set
  \begin{equation*}
    \sM^\pHecke(\sE_0,\varphi_0;y_0,z_0,\bk)
    \coloneq
    \sC^\pHecke(\sE_0,\varphi_0;y_0,z_0,\bk)/\sG^\C.
    \qedhere
  \end{equation*}
\end{definition}

The first step in the proof of \autoref{Thm_EBE=Hecke} is to show that every Hecke modification of $(\sE_0,\varphi_0)$ arises as the scattering map of a parametrized Hecke modification.

\begin{prop}
  \label{Prop_ParametrizingHeckeModifications}
  The map
  \begin{equation}
    \label{Eq_ParametrizingHeckeModifications}
    \sM^\pHecke(\sE_0,\varphi_0;y_0,z_0,\bk) \to \sM^\Hecke(\sE_0,\varphi_0;z_0,\bk)
  \end{equation}
  induced by the scattering map construction is a bijection.
\end{prop}

\begin{proof}
  Let $(\sE_1,\varphi_1;\eta)$ be a Hecke modification of $(\sE_0,\varphi_0)$ at $z_0$ of type $\bk$.
  Denote the complex vector bundles underlying $\sE_0$ and $\sE_1$ by $E_0$ and $E_1$.
  Denote the holomorphic structures on $\sE_0$ and $\sE_1$ by $\delbar_0$ and $\delbar_1$.
  The bundle $E$ is isomorphic to the bundle obtained by gluing the pullback of $E_0$ to $[0,y_0]\times \Sigma \setminus \set{(y_0,z_0)}$ and the pullback of $E_1$ to $[y_0,1]\times \Sigma \setminus \set{(y_0,z_0)}$ via $\eta$.
  Therefore,
  there is an operator $\delbar\co \Gamma(E) \to \Gamma(\Hom(\pi_\Sigma^*T\Sigma^{0,1},E))$ on $E$ whose restriction to $\set{y}\times \Sigma$ agrees with $\delbar_0$ if $y<y_0$ and with $\delbar_1$ if $y > y_0$.
  There also is a section $\varphi \in \Gamma(\pi_\Sigma^*T^*\Sigma^{1,0}\otimes \End(E))$ whose restriction to $\set{y}\times \Sigma$ agrees $\varphi_0$ if $y<y_0$ and with $\varphi_1$ if $y > y_0$.
  Define $\fd_y\co \Gamma(E) \to \Gamma(E)$ to be given by $\del_y$ on both halves of the above decomposition of $E$.
  By construction,
  $(\delbar,\varphi,\fd_y)$ is a parametrized Hecke modification and the associated scattering map induces the Hecke modification $(\sE_1,\varphi_1;\eta)$.
  This proves that the map \autoref{Eq_ParametrizingHeckeModifications} is surjective.

  Let $(\delbar,\varphi,\fd_y)$ and $(\tilde\delbar,\tilde\varphi,\tilde\fd_y)$ be two parametrized Hecke modification which induce the Hecke modifications $(\sE_1,\varphi_1;\eta)$ and $(\tilde\sE_1,\tilde\varphi_1;\tilde\eta)$.
  Suppose that the latter are isomorphic via $\zeta\co (\sE_1,\varphi_1) \to (\tilde\sE_1,\tilde\varphi_1)$.
  We can assume that both parametrized Hecke modifications are in temporal gauge.
  Therefore, on $[0,y_0)\times\Sigma$ they agree and are given by $(\delbar_0,\varphi_0,\del_y)$;
  while on $(y_0,1]\times\Sigma$
  \begin{equation*}
    (\delbar,\varphi,\fd_y) = (\delbar_1,\varphi_1,\del_y) \qandq
    (\tilde\delbar,\tilde\varphi,\tilde\fd_y) = (\tilde\delbar_1,\tilde\varphi_1,\del_y).
  \end{equation*}
  The isomorphism $\zeta$ intertwines $\delbar_1$ and $\tilde\delbar_1$ as well as $\varphi_1$ and $\tilde\varphi_1$ and commutes with the identification of $E_0$ and $E_1$ respectively $\tilde E_1$ over $\Sigma\setminus\set{z_0}$.
  Therefore, it glues with the identity on $E_0$ to a gauge transformation in $\sG^\C$ relating $(\delbar,\varphi,\fd_y)$ and $(\tilde\delbar,\tilde\varphi,\tilde\fd_y)$.
  This proves that the map \autoref{Eq_ParametrizingHeckeModifications} is injective.
\end{proof}


\section{Varying the Hermitian metric}
\label{Sec_HermitianMetrics}

The purpose of this section is to reduce \autoref{Thm_EBE=Hecke} to a uniqueness and existence result for a certain partial differential equation imposed on a Hermitian metric.

\begin{prop}
  \label{Prop_ComplexTriple=>GaugeTheoryData}
  Given a parametrized Hecke modification $(\delbar,\varphi,\fd_y)$ on $(E,H)$,
  there are unique $A_H \in \sA(E,H)$, $\phi_H \in \Omega^1(M,\fu(E,H))$, and $\xi_H \in \Omega^0(M,\fu(E,H))$ such that
  \begin{equation}
    \label{Eq_ComplexTriple=>GaugeTheoryData}
    \delbar = \nabla_{A_H}^{0,1}, \quad
    \varphi = \phi_H^{1,0}, \qandq
    \fd_y = \nabla_{A_H,\del_y} - i\xi_H.
  \end{equation}
  Moreover, $(A_H,\phi_H,\xi_H)$ has a Dirac type singularity of type $\bk$ at $(y_0,z_0)$.
\end{prop}

\begin{proof}
  This is analogous to the existence and uniqueness of the Chern connection.
  In fact, it can be reduced to it;
  see \autoref{Prop_LiftToDimensionFour}.
\end{proof}

This proposition shows that
\autoref{Thm_EBE=Hecke} is equivalent to the bijectivity of the map
\begin{equation*}
  \set*{
    (\delbar,\varphi,\fd_y) \in \sC^\pHecke(\sE_0,\varphi_0;y_0,z_0,\bk)
    :
    \text{\autoref{Eq_EBE_Metric} and } \xi_H(1,\cdot) = 0
  }/\sG
  \to
  \sM^\pHecke(\sE_0,\varphi_0;y_0,z_0,\bk).
\end{equation*}
This in turn is equivalent to the following for every parametrized Hecke modification $(\delbar,\varphi,\fd_y)$:
\begin{enumerate}
\item
  There exists a $u \in \sG^\C$ such that $u_*(\delbar,\varphi,\fd_y)$ satisfies \autoref{Eq_EBE_Metric} and $\xi_H(1,\cdot) = 0$.
\item
  The equivalence class $[u] \in \sG^\C/\sG$ is unique.
\end{enumerate}

The gauge transformed parametrized Hecke modification $u_*(\delbar,\varphi,\fd_y)$ satisfies \autoref{Eq_EBE_Metric} and $\xi_H(1,\cdot) = 0$ if and only if
with respect to gauge transformed Hermitian metric
\begin{equation*}
  K \coloneq u_*H
\end{equation*}
the parametrized Hecke modification $(\delbar,\varphi,\fd_y)$ satisfies \autoref{Eq_EBE_Metric} and $\xi_K(1,\cdot) = 0$. 
Since $K = u_*H$ depends only on $[u] \in \sG^\C/\sG$,
the preceding discussion shows that \autoref{Thm_EBE=Hecke} holds assuming the following.

\begin{prop}
  \label{Prop_ExistenceAndUniquenessOfH}
  Given $(\delbar,\varphi,\fd_y)$ a parametrized Hecke modification,  
  there exists a unique Hermitian metric of the form $K = u_*H$ with $u \in \sG^\C$ such that \autoref{Eq_EBE_Metric} and $\xi_K(1,\cdot) = 0$ hold.
\end{prop}


\section{Lift to dimension four}
\label{Sec_LiftToDimensionFour}

It will be convenient to lift the extended Bogomolny equation to dimension four,
since this allows us to directly make use of the work of \citet{Simpson1988}.

\begin{prop}
  \label{Prop_LiftToDimensionFour}
  Set
  \begin{equation*}
    X \coloneq S^1\times M.
  \end{equation*}
  Denote by $\alpha$ the coordinate on $S^1$.
  Regard $X$ as a Kähler manifold equipped with the product metric and the Kähler form
  \begin{equation*}
    \omega = \rd\alpha\wedge \rd y + \vol_\Sigma.
  \end{equation*}
  Denote by $\bE$ the pullback of $E$ to $X$.
  Given a parametrized Hecke modification $(\delbar,\varphi,\fd_y)$,
  set
  \begin{equation*}
    \bdelbar \coloneq \frac12(\del_\alpha + i\rd y \cdot \fd_y) + \delbar_E \qandq
    \bvarphi \coloneq \varphi.
  \end{equation*}
  The following hold:
  \begin{enumerate}
  \item
    \label{Prop_LiftToDimensionFour_Holomorphic}
    The operator $\bdelbar$ defines a holomorphic structure on $\bE$;
    moreover,
    \begin{equation*}
      \bdelbar\bvarphi = 0 \qandq \bvarphi\wedge\bvarphi = 0.    
    \end{equation*}
  \item
    \label{Prop_LiftToDimensionFour_Metric}
    Let $\bK$ be the pullback of a Hermitian metric $K$ on $E$.
    Denote by $A_{\bK}$ the Chern connection corresponding to $\bdelbar$ with respect to $\bK$.
    The equation \autoref{Eq_EBE_Metric} holds if and only if
    \begin{equation*}
      i\Lambda(F_{A_{\bK}} + [\bvarphi\wedge\bvarphi^{*,\bK}]) = 0.      
    \end{equation*}
  \end{enumerate}
\end{prop}

\begin{proof}
  It follows from \autoref{Eq_EBE_Holomorphic} that  
  \begin{equation*}
    \bdelbar^2 = \delbar_E^2 + i\rd y\wedge[\fd_y,\delbar_E] = 0.
  \end{equation*}
  Consequently,
  $\bdelbar$ defines a holomorphic structure.
  It also follows from \autoref{Eq_EBE_Holomorphic} that $\bdelbar\bvarphi = 0$; while $\bvarphi\wedge\bvarphi = 0$ is obvious.
  This proves \autoref{Prop_LiftToDimensionFour_Holomorphic}.

  Denote by $\pi\co X \to M$ the projection map.
  A computation shows that
  \begin{equation*}
    A_{\bK} = \pi^*A_K + \rd\alpha\wedge(\del_\alpha + \xi_K).
  \end{equation*}
  Therefore,
  \begin{equation*}
    F_{A_{\bK}} = F_{A_K} - \rd \alpha\wedge\rd y \cdot \nabla_{A_K,\del_y}\xi_K
  \end{equation*}
  and thus
  \begin{equation*}
    i\Lambda(F_{A_{\bK}} + [\bvarphi\wedge\bvarphi^{*,\bK}])
    = \pi^*\left[i\Lambda(F_{A_K} + [\varphi\wedge\varphi^{*,K}]) - i\nabla_{A_K,\del_y} \xi_K\right].
  \end{equation*}
  This proves \autoref{Prop_LiftToDimensionFour_Metric}.
\end{proof}


\section{Uniqueness of $K$}
\label{Sec_Uniqueness}

Assume the situation of \autoref{Prop_ExistenceAndUniquenessOfH}.
Given a Hermitian metric $K$ on $E$,
set
\begin{equation*}
  \fm(K) \coloneq i\Lambda(F_{A_K} + [\varphi\wedge\varphi^{*,K}]) - i\nabla_{A_K,\del_y} \xi_K.
\end{equation*}
Thus, \eqref{Eq_EBE_Metric} holds with respect to $K$ if and only if $\fm(K) = 0$.

\begin{prop}
  \label{Prop_DeltaLogTrExpS}
  For every Hermitian metric $K$ on $E$ and $s \in \Gamma(i\fu(E,K))$,
  \begin{equation*}
    \Delta \tr s = 2\tr(\fm(Ke^s) - \fm(K))
  \end{equation*}
  and
  \begin{equation*}
    \Delta \log \tr e^s \leq 2\abs{\fm(Ke^s)} + 2\abs{\fm(K)}
  \end{equation*}
  Furthermore, if $s$ is trace-free,
  then $\fm(Ke^s)$ and $\fm(K)$ can be replaced by their trace-free parts.
\end{prop}

\begin{proof}
  This follows from \cite[Lemma 3.1(c) and (d)]{Simpson1988} and \autoref{Prop_LiftToDimensionFour}.
\end{proof}

\begin{proof}[Proof of uniqueness in \autoref{Prop_ExistenceAndUniquenessOfH}]
  Suppose $K$ and $Ke^s$ are two Hermitian metrics in the $\sG^\C$--orbit of $H$ such that $\fm(K) = \fm(Ke^s) = 0$ and $\xi_K(1,\cdot) = \xi_{Ke^s}(1,\cdot) = 0$.
  It follows from the preceding proposition that $\tr s$ is harmonic and $\log \tr e^s$ is subharmonic.
  
  Since $K$ and $Ke^s$ are contained the the same $\sG^\C$--orbit,
  \begin{equation*}
    s(0,\cdot) = 0 \qandq
    \abs{s} = O(r^\alpha).
  \end{equation*}
  for some $\alpha > 0$.
  The computation proving \autoref{Prop_ComplexTriple=>GaugeTheoryData} shows that
  \begin{equation}
    \label{Eq_XiKes}
    \xi_{Ke^s}
    =
    \frac12\(\xi_K + e^{-s}\xi_Ke^s - ie^{-s}(\nabla_{A_K,\del_y}e^s)\).
  \end{equation}
  Therefore,
  \begin{equation*}
    \nabla_{A_K,\del_y} s(1,\cdot) = 0.
  \end{equation*}

  Since $\tr s$ is harmonic, bounded, vanishes at $y=0$, and satisfies Neumann boundary conditions at $y = 1$,
  it follows that $\tr s = 0$.
  Furthermore, since $\log \tr e^s$ is subharmonic,
  the above together with the maximum principle implies $\log\tr e^s \leq  \log \tr e^0 =  \log\rk E$.
  By the inequality between arithmetic and geometric means,
  \begin{equation*}
    \frac{\tr e^s}{\rk E} \geq e^{\tr s} = 1;
    \quad\text{that is:}\quad
    \log\tr e^s \geq \log\rk E
  \end{equation*}
 with equality if and only if $s = 0$.
\end{proof}


\section{Construction of $K$}
\label{Sec_Construction}

This section is devoted to the construction of $K$ using the heat flow method with boundary conditions \cite{Simpson1988,Donaldson1992}.
The analysis of its behavior at the singularity is discussed in the next section.

\begin{prop}
  \label{Prop_ConstructionOfH}
  Given a parametrized Hecke modification, $(\delbar,\varphi,\fd_y)$ on $(E,H)$,
  there exists a bounded section $s \in \Gamma(i\fu(E,H))$ such that for $K \coloneq He^s$ both $\fm(K) = 0$ and $\xi_K(1,\cdot) = 0$ hold.
\end{prop}

The proof requires the following result as a preparation.

\begin{prop}
  \label{Prop_SimpsonFlow}
  Assume the situation of \autoref{Prop_LiftToDimensionFour}.
  For $\epsilon > 0$,
  set
  \begin{equation*}
    X_\epsilon \coloneq S^1\times ([0,1]\times \Sigma \setminus B_\epsilon(y_0,z_0)).
  \end{equation*}
  Denote the pullback of $H$ to $X$ by $\bH$.
  Suppose that
  \begin{equation*}
    \Abs{i\Lambda(F_{A_{\bH}}^\circ + [\bvarphi\wedge\bvarphi^{*,\bH}])}_{L^\infty} < \infty.
  \end{equation*}
  The following hold:
  \begin{enumerate}
  \item
    \label{Prop_SimpsonFlow_LongTimeExistence}
    Let $\epsilon > 0$.
    There exists a unique solution $(\bK_t^\epsilon)_{t \in [0,\infty)}$ of
    \begin{equation}
      \label{Eq_SimpsonFlow}
      (\bK_t^\epsilon)^{-1}\del_t\bK_t^\epsilon
      =
      -i\Lambda(F_{A_{\bK_t^\epsilon}}^\circ + [\bvarphi\wedge\bvarphi^{*,\bK_t^\epsilon}])
    \end{equation}
    on $X_\epsilon$ with initial condition
    \begin{equation*}
      \bK_0^\epsilon = \bH|_{X_\epsilon}
    \end{equation*}
    and subject to the boundary conditions
    \begin{align*}
      \bK_t^\epsilon|_{S^1\times\set{0}\times \Sigma}
      &=
        \bH|_{S^1\times\set{0}\times \Sigma}, \\
      \bK_t^\epsilon|_{S^1\times\del B_\epsilon(y_0,z_0)}
      &=
        \bH|_{S^1\times \del B_\epsilon(y_0,z_0)}, \qand \\
      (\nabla_{A_\bH,\del_y}\bK_t^\epsilon)|_{S^1\times \set{1}\times \Sigma}
      &=
        0.
    \end{align*}
  \item
    \label{Prop_SimpsonFlow_Convergence}
    As $t\to \infty$, the Hermitian metrics $\bK_t^\epsilon$ converge in $C^\infty$ to a solution $\bK^\epsilon$ of
    \begin{equation*}
      i\Lambda(F_{\bK^\epsilon}^\circ + [\bvarphi\wedge\bvarphi^{*,\bK^\epsilon}]) = 0.
    \end{equation*}
  \item    
    \label{Prop_SimpsonFlow_UniformEstimates}
    The section $s_\epsilon \in \Gamma(X_\epsilon,i\su(\bE,\bH))$ defined by $\bK^\epsilon = \bH e^{s_\epsilon}$ is $S^1$--invariant and satisfies 
    \begin{equation*}
      \Abs{s_\epsilon}_{L^\infty} \lesssim 1
      \quad\text{as well as}\quad
      \Abs{s_\epsilon}_{C^k(X_\delta)} \lesssim_{k,\delta} 1
    \end{equation*}
    for every $k \in \N$ and $\delta > \epsilon$.
  \end{enumerate}
\end{prop}

\begin{proof}
  \autoref{Prop_SimpsonFlow_LongTimeExistence} follows from \citet[Section 6]{Simpson1988}.

  Set
  \begin{equation*}
    f_t \coloneq \abs{i\Lambda(F_{\bK_t}^\circ + [\bvarphi\wedge\bvarphi^*])}_{\bK_t}^2.
  \end{equation*}
  By a short computation,
  we have
  \begin{equation*}
    (\del_t+\Delta) f_t \leq 0.
  \end{equation*}
  The spectrum of $\Delta$ on $X_\epsilon$ with Dirichlet boundary conditions at $y = 0$ and at distance $\epsilon$ to the singularity as well as Neumann boundary conditions at $y = 0$ is positive.
  Therefore, there are $c,\lambda > 0$ such that
  \begin{equation*}
    \Abs{f_t}_{L^\infty} \leq ce^{-\lambda t}.
  \end{equation*}
  Consequently,
  \begin{equation*}
    \sup_{p\in X_\epsilon} \int_{0}^\infty  \sqrt{f_t} \rd t < \infty
  \end{equation*}
  This means that the path $\bK_t^\epsilon$ has finite length in the space of Hermitian metrics.
  \autoref{Prop_SimpsonFlow_Convergence} thus follows from \cite[Lemma 6.4]{Simpson1988}.
  The $S^1$--invariance of $s_\epsilon$ follows from the $S^1$--invariance of the initial condition.from \cite[Theorem 1]{Simpson1988}.

  Since $s_\epsilon$ is $S^1$--invariant and trace-free,
  by \autoref{Prop_LiftToDimensionFour} and \autoref{Prop_DeltaLogTrExpS},
  \begin{equation*}
    \Delta \log \tr(e^{s_\epsilon}) \leq 2 \abs{i\Lambda(F_{A_{\bH}}^\circ + [\bvarphi\wedge\bvarphi^{*,\bH}])}^2.
  \end{equation*}
  Let $f$ be the solution of
  \begin{equation*}
    \Delta f = 2 \abs{i\Lambda(F_{A_{\bH}}^\circ + [\bvarphi\wedge\bvarphi^{*,\bH}])}^2
  \end{equation*}
  subject to the boundary conditions
  \begin{equation*}
    f|_{S^1\times\set{0}\times\Sigma} = 0 \qandq
    \del_y f|_{S^1\times\set{1}\times\Sigma} = 0.
  \end{equation*}
  Choose a constant $c$ such that $f+c > 0$.
  Set
  \begin{equation*}
    g \coloneq \log\tr(e^{s_\epsilon}) - (f+c).
  \end{equation*}
  The function $g$ is subharmonic on $X_\epsilon$.
  Thus it achieves its maximum on the boundary.  
  On $S^1\times \del B_\epsilon(y_0,z_0)$ and $S^1\times\set{0}\times\Sigma$,
  the function $g$ is negative.
  At $S^1\times \set{1}\times \Sigma$, $\del_y f = 0$.
  By the reflection principle, the maximum is not achieved at $y=1$ unless $g$ is constant.
  It follows that $g \leq 0$.
  This shows that $\abs{\log\tr(e^{s_\epsilon})}$ is bounded independent of $\epsilon$.
  Since $s$ is trace-free,
  it follows that $\abs{s_\epsilon}$ is bounded independent of $\epsilon$.
  By \cite[Lemma 6.4]{Simpson1988}, which is an extension of \cite[Lemma 19]{Donaldson1985} with boundary conditions, and elliptic bootstrapping the asserted $C^k$ bounds on $s_\epsilon$ follow.
\end{proof}

\begin{proof}[Proof of \autoref{Prop_ConstructionOfH}]
  Without loss of generality we can assume that $H$ is such that $\xi_H$ vanishes at $y = 1$.
  
  There is a unique $f \in C^\infty([0,1]\times\Sigma\setminus\set{y_0,z_0})$ which satisfies
  \begin{equation*}
    \frac12 \Delta f = \tr(i\Lambda F_{A_{H}} -i\nabla_{A_{H},\del_y}\xi_{H}),
  \end{equation*}
  is bounded, vanishes at $y = 0$, and satisfies Neumann boundary conditions at $y=0$.
  A barrier argument shows that $\abs{f} = O(r^\alpha)$ for some $\alpha > 0$.
  Replacing $H$ with $He^f$,
  we may assume that
  \begin{equation*}
    \tr(i\Lambda F_{A_H} -i\nabla_{A_H,\del_y}\xi_H) = 0.
  \end{equation*}
  
  For every $s \in \Gamma(i\su(E,H))$, the above condition holds for $He^s$ instead of $H$ as well.
  Let $s_\epsilon$ be as in \autoref{Prop_SimpsonFlow}.
  Take the limit of $s_\epsilon$ on each $X_\delta$ as first $\epsilon$ tends to zero and then $\delta$ tends to zero.
  This limit is the pullback of a section $s$ defined over $[0,1]\times\Sigma\setminus\set{y_0,z_0}$ which has the desired properties.
  Since $\nabla_{A_H,\del_y} s$ vanishes at $y = 1$,
  it follows from \autoref{Eq_XiKes} that $\xi_K$ vanishes at $y = 1$.
\end{proof}


\section{Singularity analysis}
\label{Sec_SingularityAnalysis}

It remains to analyze the section $s$ constructed via \autoref{Prop_ConstructionOfH} near the singularity.
The following result completes the proof of \autoref{Prop_ExistenceAndUniquenessOfH} and thus \autoref{Thm_EBE=Hecke}.

\begin{prop}
  \label{Prop_SingularityAnalysis}
  Consider the unit ball $B \subset \R\times \C$ with a metric $g = g_0 + O(r^2)$.
  Set $\dot B \coloneq B\setminus\set{0}$.
  Let $\bk \in \Z^r$ be such that \autoref{Eq_KIncreasing} and let $\alpha > 0$.
  Let $(\delbar,\phi,\fd_y)$ be a parametrized Hecke modification on $(E_\bk,H_\bk)$.
  If $s \in \Gamma(i\fu(E_\bk,H_\bk))$ is bounded and satisfies
  \begin{equation*}
    \fm(H_\bk e^s) = 0,
  \end{equation*}
  then there is an $\alpha > 0$ and $s_0 \in \Gamma(i\fu(E_\bk,H_\bk))$ such that
  \begin{equation*}
    \nabla_{A_\bk} s_0 = 0 \qandq [\xi_\bk,s_0] = 0
  \end{equation*}
  and for every $k \in \N_0$
  \begin{equation*}
    \nabla_{A_\bk}^k(s - s_0) = O(r^{-k+\alpha});
  \end{equation*}
  that is: $H_\bk e^s = e^{s/2}_*H_\bk$ is in the $\sG^\C$--orbit of $H_\bk$.
\end{prop}

The proof of this result uses the technique developed in \cite{Jacob2016a}.
Henceforth,
we shall assume the situation of \autoref{Prop_SingularityAnalysis}.
Moreover, we drop the subscript $\bk$ from $E_\bk$ and $H_\bk$ to simplify notation.

Define $\fV \co \Gamma(i\fu(E,H)) \to \Omega^1(\dot B,i\fu(E,H))\times \Gamma(i\fu(E,H))$ by
\begin{equation*}
  \fV s \coloneq (\nabla_{A_\bk}s,[\xi_\bk,s])
\end{equation*}
The following a priori Morrey estimate is the crucial ingredient of the proof of \autoref{Prop_SingularityAnalysis}. 

\begin{prop}
  \label{Prop_AprioriMorreyEstimate}
  For some $\alpha > 0$,
  we have
  \begin{equation*}
    \int_{B_r} \abs{\fV s}^2 \lesssim r^{1+2\alpha}.
  \end{equation*}
\end{prop}

\begin{proof}[Proof of \autoref{Prop_SingularityAnalysis} assuming \autoref{Prop_AprioriMorreyEstimate}]
  Denote by $s_r$ the pullback of $s$ from $B_r$ to $B$.
  By \autoref{Prop_AprioriMorreyEstimate},
  \begin{equation*}
    \Abs{\nabla_{A_\bk} s_r}_{L^2(B)} + \Abs{[\xi_\bk,s_r]}_{L^2(B)} \lesssim r^\alpha.
  \end{equation*}
  Denote by $\fm_r$ the map $\fm$ with respect to $r^{-2}$ times the pullback of the Riemannian metric and the parametrized Hecke modification from $B_r$ to $B$.
  The equation $\fm_r(H e^{s_r}) = 0$ can be written schematically as
  \begin{equation*}
    \nabla_{A_H}^*\nabla_{A_H}s_r + B(\nabla_{A_H}s\otimes \nabla_{A_H}s_r) = C(\fm_r(H))
  \end{equation*}
  where $B$ and $C$ are linear with coefficients depending only on $s$, but not its derivatives.

  Set
  \begin{equation*}
    a \coloneq \nabla_{A_H} - \nabla_{A_\bk}, \quad
    \hat\phi = \phi_H - \phi_\bk, \qandq
    \hat \xi \coloneq \xi_H - \xi_\bk.
  \end{equation*}
  It follows from \autoref{Eq_ParametrizedHeckeModification_Estimates} that, after possibly decreasing the value of $\alpha > 0$, for $k \in \N_0$
  \begin{equation}
    \label{Eq_AHatXi}
    \nabla_{A_\bk}^k a = O(r^{-k-1+\alpha}), \quad
    \nabla_{A_\bk}^k\hat\phi = O(r^{-k}), \qandq
    \nabla_{A_\bk}^k\hat\xi = O(r^{-k-1+\alpha}).
  \end{equation}
  Therefore, $\fm_r(H) = O(r^\alpha)$ on $B \setminus B_{1/8}$.
  
  As in \cite[Section 5]{Jacob2016a},
  it follows from Bando--Siu's interior estimates \cites[Proposition 1]{Bando1994}[Theorem C.1]{Jacob2016} that for $k \in \N_0$
  \begin{equation*}
    \Abs{\nabla_{A_\bk} s_r}_{C^k(B_{1/2}\setminus B_{1/4})} + \Abs{[\xi_\bk,s_r]}_{C^k(B_{1/2}\setminus B_{1/4})} \lesssim_k r^\alpha.
  \end{equation*}
  Consequently,
  there is an $s_0 \in \ker \fV$ such that for $k \in \N_0$
  \begin{equation*}
    \Abs{\nabla_{A_\bk}^k(s_r - s_0)}_{L^\infty(B_{1/2}\setminus B_{1/4})} \lesssim_k r^\alpha.
  \end{equation*}
  This translates to the asserted estimates for $s$.
\end{proof}

The proof of \autoref{Prop_AprioriMorreyEstimate} occupies the remainder of this section.

\subsection{A Neumann--Poincaré inequality}

Denoting the radial coordinate by $r$,
we can write
\begin{equation*}
  \fV s \coloneq (\rd r\cdot \nabla_{\del_r}s,\fV_rs)
\end{equation*}
for a family of operators $\fV_{r} \co \Gamma(\del B_r,i\fu(E,H)) \to \Omega^1(\del B_r,i\fu(E,H))\times \Gamma(\del B_r,i\fu(E,H))$.
The pullback of $\fV_r$ to $\del B$ agrees with $\fV_1$.
Consequently, we can identify
\begin{equation*}
  \ker \fV_r = \ker \fV_1 \eqcolon N.
\end{equation*}
Denote by $\pi_r \co \Gamma(\del B_r,i\fu(E,H)) \to N$ the $L^2$--orthogonal projection onto $N$.
Set
\begin{equation*}
  \Pi_r s \coloneq \frac1r\int_r^{2r} \pi_t(s) \,\rd t.
\end{equation*}

\begin{prop}
  \label{Prop_NeumannPoincare}
  For every $s \in \Gamma(i\fu(E,H))$ and $r \in [0,1/2]$,
  we have
  \begin{equation}
    \label{Eq_NeumannPoincare}
    \int_{B_{2r}\setminus B_r} \abs{s-\Pi_r s}^2
    \lesssim
    r^2\int_{B_{2r}\setminus B_r} \abs{\fV s}^2.
  \end{equation}
\end{prop}

\begin{proof}
  The proof is identical to that of \cite[Proposition 4.2]{Jacob2016a}.
  For the readers convenience we will reproduce the argument here.

  Since \autoref{Eq_NeumannPoincare} is scale invariant,
  we may assume $r=1/2$.
  Furthermore, it suffices to prove the cylindrical estimate
  \begin{equation*}
    \int_{1/2}^1 \int_{\del B} \abs{s(t,\hat x) - \Pi s(t,\cdot)}^2 \,\rd \hat x \rd t
    \lesssim
    \int_{1/2}^1 \int_{\del B} \abs{\del_t s(t,\hat x)}^2 + \abs{\fV_1 s(t,\hat x)}^2 \,\rd \hat x \rd t
  \end{equation*}
  with $s$ denoting a section over $[1/2,1]\times \del B$,
  \begin{equation*}
    \pi \coloneq \pi_1, \qandq
    \Pi s \coloneq 2\int_{1/2}^1 \pi s(t,\cdot) \, \rd t.
  \end{equation*}

  To prove this inequality,
  we compute
  \begin{align*}
    &\int_{1/2}^1 \int_{\del B} \abs{s(t,\hat x) - \Pi s(t,\cdot)}^2 \rd \hat x \rd t\\
    &\quad=
      4\int_{1/2}^1 \int_{\del B} \abs*{\int_{1/2}^1 s(t,\hat x) -  \pi s(u,\cdot) \,\rd u}^2 \,\rd \hat x \rd t \\
    &\quad\lesssim
      \int_{1/2}^1\int_{1/2}^1\int_{\del B}  \abs*{s(t,\hat x) -  \pi s(u,\cdot)}^2 \,\rd \hat x\rd u\rd t \\
    &\quad\lesssim
      \int_{1/2}^1\int_{1/2}^1\int_{\del B} \abs*{s(t,\hat x) -  \pi s(t,\cdot)}^2 + \abs*{\pi s(t,\cdot) -  \pi s(u,\cdot)}^2 \,\rd \hat x\rd u\rd t.
  \end{align*}
  The first summand can be bounded as follows
  \begin{align*}
    \int_{1/2}^1\int_{1/2}^1\int_{\del B} \abs*{s(t,\hat x) -  \pi s(t,\cdot)}^2 \,\rd\hat x\rd t\rd u
    &\lesssim
      \int_{1/2}^1\int_{1/2}^1\int_{\del B} \abs*{\fV_1 s(t,\hat x)}^2 \,\rd\hat x\rd t\rd u \\ 
    &\lesssim
      \int_{1/2}^1 \int_{\del B} \abs*{\fV_1 s(t,\hat x)}^2 \,\rd\hat x\rd t.
  \end{align*}
  The second summand can be controlled as in the usual proof of the Neumann--Poincare inequality:
  We have
  \begin{align*}
    \abs*{\pi s(t,\cdot) -  \pi s(u,\cdot)}
    &=
      \abs*{\int_0^1 \del_v \pi s(t + v(t-u),\cdot) \,\rd v} \\
    &\leq
      \abs*{\int_0^1  \pi (\del_t s)(t + v(t-u),\cdot) \,\rd v} \\
    &\lesssim
      \(\int_0^1 \int_{\del B} \abs{(\del_t s)(t + v(t-u),\hat x)}^2 \,\rd\hat x\rd v\)^{1/2}.
  \end{align*}
  Plugging this into the second summand and symmetry considerations yield
  \begin{align*}
    &\int_{1/2}^1\int_{1/2}^1\int_{\del B} \abs*{\pi s(t,\cdot) -  \pi s(u,\cdot)}^2 \,\rd \hat x\rd u\rd t \\
    &\quad\lesssim
      \int_{1/2}^1 \int_{1/2}^1 \int_0^1 \int_{\del B} \abs{(\del_t s)(t + v(t-u),\hat x)}^2 \,\rd\hat x\rd v\rd u\rd t \\
    &\quad\lesssim
      \int_{1/2}^1 \int_{\del B} \abs{\del_t s(t,\hat x)}^2 \,\rd \hat x\rd t.
  \end{align*}
  This finishes the proof.
\end{proof}

\subsection{A differential inequality}

The following differential inequality for
\begin{equation*}
  \hat s_r \coloneq \log\paren[\big]{e^{-\Pi_r s}e^s}.
\end{equation*}
lies at the heart of the proof of \autoref{Prop_AprioriMorreyEstimate}.
By construction,
the section $\hat s_r$ is self-adjoint with respect to $He^s$ as well as $He^{\Pi_r s}$,
and
\begin{equation*}
  H e^s = \paren[\big]{H e^{\Pi_r s}}e^{\hat s_r}.
\end{equation*}

\begin{prop}
  \label{Prop_VHatSR}
  The section $\hat s_r$ satisfies
  \begin{equation*}
    \abs{\fV s}
    \lesssim
      \abs{\fV \hat s_r}, \quad
    \abs{\hat s_r}
    \lesssim
      \abs{s - \Pi_rs}, \qandq
    \abs{\fV \hat s_r}^2
    \lesssim
      r^{-2+\beta} - \Delta\abs{\hat s_r}^2
  \end{equation*}
  for some $\beta > 0$.
\end{prop}

The proof relies on the following identity.

\begin{prop}
  \label{Prop_(KHes-KH,s)}
  We have
  \begin{equation*}
    \inner{\fm(He^s)-\fm(H)}{s}
    = \frac14\Delta\abs{s}^2
    + \frac12\abs{\upsilon(-s)\nabla_{A_H} s}^2
    + \frac12\abs{\upsilon(-s)[\phi_H,s]}^2
    + \frac12\abs{\upsilon(-s)[\xi_H,s]}^2
  \end{equation*}
  with
  \begin{equation*}
    \upsilon(s) = \sqrt{\frac{e^{\ad_s}-\id}{\ad_s}} \in \End(\gl(E)).
  \end{equation*}
\end{prop}

\begin{proof}
  We prove the analogous formula in dimension four.
  We have
  \begin{equation*}
    \del_{A_{He^s}} = e^{-s}\del_{A_H} e^s = \del_H + \Upsilon(-s)\del_H s
    \qandq
    \varphi^{*,He^s} = e^{-s}\varphi^{*,H} e^{s}
  \end{equation*}
  with
  \begin{equation*}
    \Upsilon(s) = \frac{e^{\ad_s}-\id}{\ad_s}.
  \end{equation*}
  Set
  \begin{equation*}
    D \coloneq \delbar + i\varphi \qandq \bar D_H \coloneq \del_H - i\varphi^{*,H}.
  \end{equation*}
  The above formula asserts that
  \begin{equation*}
    \bar D_{He^s} = e^{-s}\bar D_H e^{s} = \bar D_H + \Upsilon(-s) \bar D_Hs.
  \end{equation*}
  Since
  \begin{equation*}
    D+\bar D_H = \nabla_{A_H} + i\phi_H,
  \end{equation*}
  we have
  \begin{equation*}
    \fm(H) = \frac12i\Lambda [D,\bar D_H].
  \end{equation*}
  Therefore,
  \begin{align*}
    \inner{\fm(He^s) - \fm(H)}{s}
    &=
      i\Lambda \inner{D (\Upsilon(-s) \bar D_Hs)}{s} \\
    &=
      i\Lambda \delbar \inner{\Upsilon(-s) \bar D_Hs}{s}
      + i\Lambda \Inner{\Upsilon(-s) \bar D_Hs \wedge \bar D_H s} \\
    &=
      \del^* \inner{\bar D_Hs}{\Upsilon(s) s} + \abs{\upsilon(-s) \bar D_Hs}^2 \\
    &=
      \frac12\del^*\del \abs{s}^2 + \abs{\upsilon(-s) \bar D_Hs}^2 \\
    &=
      \frac14\Delta \abs{s}^2 + \frac12\abs{\upsilon(-s) (\nabla_H + i[\phi,\cdot]) s}^2.
      \qedhere
  \end{align*}  
\end{proof}

\begin{proof}[Proof of \autoref{Prop_VHatSR}]
  The first two estimates are elementary.
  To prove the last estimate we argue as follows.  
  Set
  \begin{equation*}
    a \coloneq \nabla_{A_H} - \nabla_{A_\bk} \qandq
    \hat \xi \coloneq \xi_H - \xi_\bk.
  \end{equation*}
  By \autoref{Eq_AHatXi} and since $\Pi_r s$ lies in the kernel of $\fV$,
  for some $\beta > 0$
  \begin{align*}
    \abs{\fV \hat s_r}^2
    &\lesssim
      \abs{\nabla_{A_H} \hat s_r}^2 + \abs{[\xi_H,\hat s_r]}^2 + r^{-2+2\beta} \\
    &\lesssim
      \abs{\nabla_{A_{He^{\Pi_rs}}\hat s_r}}^2 + \abs{[\xi_{He^{\Pi_rs}},\hat s_r]}^2 + r^{-2+2\beta}.
  \end{align*}
  Therefore, it suffices to estimate $\abs{\nabla_{A_{He^{\Pi_rs}}\hat s_r}}^2 + \abs{[\xi_{He^{\Pi_rs}},\hat s_r]}^2$.
  
  Since $\hat s_r$ is bounded,
  $\upsilon(\hat s_r)$ is bounded away from zero.
  Hence, by \autoref{Prop_(KHes-KH,s)} with $He^{\Pi_r s}$ instead of $H$ and $\hat s_r$ instead of $s$,
  \begin{equation*}
    \abs{\nabla_{A_{He^{\Pi_r s}}}\hat s_r}^2 + \abs{[\phi_{He^{\Pi_r}},\hat s_r]}^2
    \lesssim
    \abs{\fm(He^s)} + \abs{\fm(He^{\Pi_r s})} - \Delta \abs{\hat s_r}^2.
  \end{equation*}
  It follows from \autoref{Eq_AHatXi} that $\abs{\fm(H)} = O(r^{-2+\beta})$.
  Moreover, since $\Pi_rs$ lies in the kernel of $\fV$,
  $\abs{\fm(He^{\Pi_r s})} = O(r^{-2+\beta})$.
  Furthermore, $\fm(He^s) = 0$,
  Putting all of the above together yields the asserted estimate.
\end{proof}

\subsection{Proof of \autoref{Prop_AprioriMorreyEstimate}}

Set
\begin{equation*}
  g(r) \coloneq \int_{B_r} \abs{x}^{-1}\abs{\fV s}^2
\end{equation*}
with $\abs{x}$ denoting the distance to the center of the ball $B_r$.
The upcoming three steps show that $g(r) \lesssim r^{2\alpha}$ for some $\alpha > 0$.
This implies the assertion.

\begin{step}
  The function $\abs{x}^{-1}\abs{\fV s}^2$ is integrable;
  in particular: $g \leq c$.  
\end{step}

Fix a smooth function $\chi \co [0,\infty) \to [0,1]$ which is equal to one on $[0,1]$ and vanishes outside $[0,2]$.
Set $\chi_r(\cdot) \coloneq  \chi(\abs{\cdot}/r)$.
Denote by $G$ the Green's function of $B$ centered at $0$.  
For $r > \epsilon > 0$, using \autoref{Prop_VHatSR}, we have
\begin{align*}
  \int_{B_r\setminus B_\epsilon} \abs{x}^{-1}\abs{\fV s}^2
  &\lesssim
    \int_{B_{2r}\setminus B_{\epsilon/2}} \chi_r(1-\chi_{\epsilon/2})G(r^{-2+\beta} - \Delta \abs{\hat s_r}^2) \\
  &\lesssim
    r^{\beta}
    + r^{-3}\int_{B_{2r}\setminus B_r} \abs{\hat s_r}^2
    + \epsilon^{-3}\int_{B_\epsilon\setminus B_{\epsilon/2}} \abs{\hat s_r}^2.
\end{align*}
Since $s$ is bounded, the right-hand side is bounded independent of $\epsilon$.
This proves the integrability of $\abs{x}^{-1}\abs{\fV s}^2$ and the yields a bound on $g$.

\begin{step}
  There are constants $\gamma \in [0,1)$ and $c > 0$ such that
  \begin{equation*}
    g(r) \leq \gamma g(2r) + cr^{\beta}.
  \end{equation*}    
\end{step}

Continue the inequality from the previous step using the Neumann--Poicaré estimate \autoref{Eq_NeumannPoincare} as
\begin{align*}
  \int_{B_r\setminus B_\epsilon} \abs{x}^{-1}\abs{\fV s}^2
  &\lesssim
    r^\beta
    + r^{-3}\int_{B_{2r}\setminus B_r} \abs{s - \Pi_rs}^2
    + \epsilon^{-3}\int_{B_\epsilon\setminus B_{\epsilon/2}} \abs{s - \Pi_rs}^2 \\
  &\lesssim
    r^\beta
    + r^{-1}\int_{B_{2r}\setminus B_r} \abs{\fV s}^2
    + \epsilon^{-1}\int_{B_\epsilon\setminus B_{\epsilon/2}} \abs{\fV s}^2 \\
  &\lesssim
    r^\beta + g(2r) - g(r) + g(\epsilon).
\end{align*}
By Lebesgue's monotone convergence theorem,
the last term vanishes as $\epsilon$ tends to zero.
Therefore,
\begin{equation*}
  g(r) \lesssim g(2r) - g(r) + r^\beta
\end{equation*}

\begin{step}
  For some $\alpha > 0$, $g \lesssim r^{2\alpha}$.
\end{step}

This follows from the preceding steps by an elementary argument;
see, e.g., \cite[Step 3 in the proof of Proposition C.2]{Jacob2016}.
\qed


\appendix

\section{Sequences of Hecke modifications}
\label{Sec_SequencesOfHeckeModifications}

This appendix discusses the extension of \autoref{Thm_EBE=Hecke} to sequences of Hecke modifications.
Let $\Sigma$ be a closed Riemann surface,
let $(\sE_0,\varphi_0)$ be a Higgs bundle over $\Sigma$ of rank $r$,
let $z_1,\ldots,z_n \in \Sigma$,
and let $\bk_1,\ldots,\bk_n \in \Z^r$ satisfying \autoref{Eq_KIncreasing}.

\begin{definition}
  A \defined{sequence of Hecke modifications} of $(\sE_0,\varphi_0)$ at $z_1,\ldots,z_n$ of type $\bk_1,\ldots,\bk_n$ consists of a Hecke modification
  \begin{equation*}
    \eta_i \co (\sE_{i-1},\varphi_{i-1})|_{\Sigma\setminus\set{z_i}} \iso (\sE_i,\varphi_i)|_{\Sigma\setminus\set{z_i}}
  \end{equation*}
  at $z_i$ of type $\bk_i$ for every $i = 1,\ldots,n$.
  An \defined{isomorphism} between two sequences of Hecke modification $(\sE_i,\varphi_i;\eta_i)_{i=1}^n$ and $(\tilde\sE_i,\tilde\varphi_i;\tilde\eta_i)_{i=1}^n$ consists of an isomorphism
  \begin{equation*}
    \zeta_i\co (\sE_i,\varphi_i) \to (\tilde\sE_i,\tilde\varphi_i)
  \end{equation*}
  of Higgs bundles such that
  \begin{equation*}
    \zeta_{i-1}\eta_i = \tilde\eta_i\zeta_i
  \end{equation*}
  for every $i = 1,\ldots,n$ and with $\zeta_0 \coloneq \id_{\sE_0}$.
  We denote by
  \begin{equation*}
    \sM^\Hecke(\sE_0,\varphi_0;z_1,\ldots,z_n,\bk_1,\ldots,\bk_n)
  \end{equation*}
  the set of all isomorphism classes of sequences of Hecke modifications of $(\sE_0,\varphi_0)$  at $z_1,\ldots,z_n$ of type $\bk_1,\ldots,\bk_n$.
\end{definition}

Denote by $E_0$ the complex vector bundle underlying $\sE_0$.
Henceforth, we assume that $H_0$ is a Hermitian metric on $E_0$.
Furthermore, fix
\begin{equation*}
  0 < y_1 < y_2 < \cdots < y_n < 1.
\end{equation*}
As in \autoref{Prop_ConstructE},
there exists a Hermitian vector bundle $(E,H)$ over
\begin{equation*}
  M \coloneq [0,1]\times\Sigma \setminus \set{(y_1,z_1),\ldots,(y_n,z_n)}
\end{equation*}
together with a framing $\Psi_i$ at $(y_i,z_i)$ of type $\bk_i$ for every $i=1,\ldots,n$.
Any two choices of $(E,H;\Psi_1,\ldots,\Psi_n)$ are isomorphic.
Throughout the remainder of this appendix, we fix one such choice.

\begin{definition}
  Denote by $\sC^\EBE(A_0,\phi_0;y_1,z_1,\ldots,y_n,z_n,\bk_1,\ldots,\bk_n)$ the set of triples
  \begin{equation*}
    A \in \sA(E,H), \quad
    \phi \in \Omega^1(M,\fu(E,H)),
    \qandq
    \xi \in \Omega^0(M,\fu(E,H))
  \end{equation*}
  satisfying the extended Bogomolny equation \autoref{Eq_EBE}, as well as
  \begin{equation*}
    i(\del_y)\phi = 0,
  \end{equation*}
  and the boundary conditions
  \begin{equation*}
    A|_{\set{0}\times\Sigma} = A_0, \quad
    \phi|_{\set{0}\times\Sigma} = \phi_0, \qandq
    \xi|_{\set{1}\times\Sigma} = 0.
  \end{equation*}
  Denote by
  \begin{equation*}
    \sG \subset \sG(E,H)
  \end{equation*}
  the subgroup of unitary gauge transformations of $(E,H)$ which are singularity preserving at $(y_1,z_1),\ldots,(y_n,z_n)$ and restrict to the identity on $\set{0}\times \Sigma$.
  Set
  \begin{equation*}
    \sM^\EBE(A_0,\phi_0;y_1,z_1,\ldots,y_n,z_n,\bk_1,\ldots,\bk_n)
    \coloneq
    \sC^\EBE(A_0,\phi_0;y_1,z_1,\ldots,y_n,z_n,\bk_1,\ldots,\bk_n)/\sG.
    \qedhere
  \end{equation*}
\end{definition}

Let $(A,\phi,\xi) \in \sC^\EBE(A_0,\phi_0;y_1,z_1,\ldots,y_n,z_n,\bk_1,\ldots,\bk_n)$.
Let
\begin{equation*}
  y_1 < m_1 < y_2 < m_2 < \ldots < y_n < m_n \coloneq 1.
\end{equation*}
The scattering map construction from \autoref{Sec_ScatteringMap} restricted to $[0,m_1]\times\Sigma$ yields a Hecke modification $(\sE_1,\varphi_1;\eta_1)$ of $(\sE_0,\varphi_0)$ at $z_1$ of type $\bk_1$.
Similarly, we obtain a Hecke modification $(\sE_i,\varphi_i;\eta_1)$ of $(\sE_{i-1},\varphi_{i-1})$ at $z_i$ of type $\bk_i$ for every $i=1,\ldots,n$.
A different choice of $\tilde m_i \in (y_i,y_{i+1})$ may yield a different Hecke modification $(\tilde\sE_i,\tilde\varphi_i;\tilde\eta_i)$.
However, these Hecke modifications are isomorphic via the scattering map from $m_i$ to $\tilde m_i$.
Therefore, we obtain a map
\begin{equation*}
  \sC^\EBE(A_0,\phi_0;y_1,z_1,\ldots,y_n,z_n,\bk_1,\ldots,\bk_n) \to \sM^\Hecke(\sE_0,\varphi_0;z_1,\ldots,z_n,\bk_1,\ldots,\bk_n).
\end{equation*}
This map is $\sG$--invariant.
We have the following extension of \autoref{Thm_EBE=Hecke}.

\begin{theorem}
  \label{Thm_EBE=Hecke_Sequences}
  The map
  \begin{equation*}
    \sM^\EBE(A_0,\phi_0;y_1,z_1,\ldots,y_n,z_n,\bk_1,\ldots,\bk_n) \to     \sM^\Hecke(\sE_0,\varphi_0;z_1,\ldots,z_n,\bk_1,\ldots,\bk_n)
  \end{equation*}
  induced by the scattering map construction is a bijection.
\end{theorem}

\begin{proof}
  The proof is essentially the same as that of \autoref{Thm_EBE=Hecke}.
  The notion of parametrized Hecke modifications can be extended to parametrized sequences of Hecke modifications yielding a moduli space $\sM^\pHecke(\sE_0,\varphi_0;y_1,z_1,\ldots,y_n,z_n,\bk_1,\ldots,\bk_n)$.
  As in the proof of \autoref{Prop_ParametrizingHeckeModifications},
  one shows that the scattering map yields a bijection
  \begin{equation*}
    \sM^\pHecke(\sE_0,\varphi_0;y_1,z_1,\ldots,y_n,z_n,\bk_1,\ldots,\bk_n)
    \to
    \sM^\Hecke(\sE_0,\varphi_0;z_1,\ldots,z_n,\bk_1,\ldots,\bk_n).
  \end{equation*}
  Finally, the arguments from \autoref{Sec_HermitianMetrics}, \autoref{Sec_LiftToDimensionFour}, \autoref{Sec_Uniqueness}, \autoref{Sec_Construction}, and \autoref{Sec_SingularityAnalysis} show that the obvious map
  \begin{equation*}
    \sM^\EBE(A_0,\phi_0;y_1,z_1,\ldots,y_n,z_n,\bk_1,\ldots,\bk_n) \to     \sM^\pHecke(\sE_0,\varphi_0;y_1,z_1,\ldots,y_n,z_n,\bk_1,\ldots,\bk_n)
  \end{equation*}
  is a bijection.
\end{proof}

\begin{remark}
  If $\varphi = 0$, then the above reduces to the notion of a sequence of Hecke modifications of a holomorphic vector bundle;
  see, e.g., \cites[Section 1.5.1]{Wong2013}[Section 2.4]{Boozer2018}.
\end{remark}


\printreferences

\end{document}
